\documentclass[11pt]{article}
\usepackage{graphicx}
\usepackage{fullpage}
\usepackage{amsmath}
\usepackage{amssymb}
\usepackage{verbatim} 

\usepackage{tikz}

\usepackage{ulem} 

\usepackage{bbm} 
\usepackage{amsthm} 

\newtheorem{thm}{Theorem}[section]
\newtheorem{defn}[thm]{Definition}
\newtheorem{lem}[thm]{Lemma}

\newtheorem{conj}[thm]{Conjecture}

\title{Some Bounds on the Rainbow Connection Number of 3-, 4- and 5-connected Graphs}
\author{Irene Y. Lo}

\begin{document}

\maketitle
\normalem

\begin{abstract}
The rainbow connection number, $rc(G)$, of a connected graph $G$ is the minimum
number of colors needed to color its edges so that every pair of vertices is connected
by at least one path in which no two edges are colored the same. We show that for $\kappa =3$ or $\kappa = 4$, every $\kappa$-connected graph $G$ on $n$ vertices with diameter $\frac{n}{\kappa}-c$ satisfies $rc(G) \leq \frac{n}{\kappa} + 15c + 18$. We also show that for every maximal planar graph $G$, $rc(G) \leq \frac{n}{\kappa} + 36$. This proves a conjecture of Li et al. for graphs with large diameter and maximal planar graphs.
\end{abstract}

\section{Introduction}
Let $G$ be a simple, undirected, connected graph on $n$ vertices, such that its edges are colored by some edge coloring $c$. We say that a path $P$ in $G$ is a \textit{rainbow path} if no two edges of $P$ are the same color. We say that the edge-colored graph $(G, c)$ is \textit{rainbow-connected} if every pair of vertices is connected by a rainbow path, and that the coloring $c$ is a \textit{rainbow coloring} of the graph $G$. The \textit{rainbow connection number} of a graph $G$, $rc(G)$, is the minimum number of colors required to rainbow color $G$. For example, the rainbow connection number of a complete graph is $1$, that of a path, or in general, any tree, is $n-1$, and that of a cycle is $\lceil \frac{n}{2} \rceil$. A basic introduction to the subject can be found in Chapter 11 of \cite{surveybook}.

The concept of rainbow coloring was introduced by Chartrand, Johns, McKeon and Zhang in 2008 \cite{chartrand}. Computing the rainbow connection number of a graph was later shown to be NP-hard by Chakraborty, Fischer, Matsligh and Yuster \cite{chakraborty}. However, the rainbow connection number is still of interest as a `quantifiable' extension of the concept of connectivity in a graph \cite{caro, chakraborty}. In particular, there has been much interest in finding tight upper bounds for the rainbow connection number in terms of other measures of connectivity.

Recent results presented in the literature are as follows. Basavaraju et al. proved an upper bound in terms of radius, $rc(G) \leq r(r+2)$, and showed that it is tight \cite{bas}. Chandran et al. proved an upper bound in terms of minimum degree, $rc(G) \leq 3n/(\delta+1)+3$, and showed that it is tight up to additive factors \cite{chandran}. Upper bounds in terms of vertex and edge connectivity, $rc(G) \leq 3n/(\kappa+1)+3$ and $rc(G) \leq 3n/(\lambda+1)+3$, follow trivially, and Li et al. showed that the bound in terms of edge connectivity is tight up to additive factors for infinitely many values of $n$ and $\lambda$ \cite{li12}. Improving the bound in terms of vertex connectivity, Ekstein et al. and Li et al. independently proved that $rc(G) \leq \left\lceil\frac{n}{2}\right\rceil$ for 2-connected graphs, and that the bound is tight \cite{ekstein, li12}. In addition, Li et al. proved that $rc(G) \leq (2 + \varepsilon) \frac{n}{\kappa} + \frac{23}{\varepsilon^2}$ for any $\varepsilon > 0 $, so that $rc(G) \leq \frac{n}{\kappa} + C_0$ for graphs of high girth \cite{li12}. This led them to make the following conjecture.

\begin{conj}[Li et al. \cite{li12}]\label{mainconj}
Let $G$ be a $\kappa$-connected graph on $n$ vertices. Then there exists a constant $C$ such that
\begin{eqnarray*}
rc(G) < \frac{n}{\kappa} + C.
\end{eqnarray*}
\end{conj}

\noindent We prove Conjecture \ref{mainconj} for graphs with large diameter and maximal planar graphs.

\section{Preliminary Definitions and Results}

\begin{defn}
\end{defn} \vspace{-23pt}$~~~~~~~~~~~~~~$ Let $G$ be a simple, connected, undirected graph. $G$ is \textit{$k$-connected}, or \textit{$k$-vertex-connected}, if the removal of less than $k$ vertices does not disconnect the graph. Similarly, $G$ is \textit{$k$-edge-connected} if the remove of less than $k$ edges does not disconnect the graph. The \textit{vertex connectivity} of a graph, $\kappa(G)$, is the maximum value of $k$ such that $G$ is $k$-connected, and the \textit{edge connectivity}, $\lambda(G)$, is the maximum value of $k$ such that $G$ is $k$-edge-connected. Given a set $X \subset V(G)$, the graph $G|_X$ is the subgraph of $G$ induced by the set $X$, with vertex set $X$ and edge set given by the subset of edges of $G$ with both ends in $X$. The \textit{degree} of a vertex $deg(v)$ is the number of other vertices adjacent to it, and the \textit{minimum degree of $G$} is $\delta(G): = \min_{v \in V(G)} deg(v)$. A \textit{leaf} of a graph is a vertex with degree $1$.

\begin{defn}
\end{defn} \vspace{-23pt}$~~~~~~~~~~~~~~$ The \textit{length} of a path or cycle is the number of edges in the path or cycle. The \textit{girth} of a graph is the length of shortest cycle in the graph. The \textit{distance} between two vertices $u$ and $v$ of $G$, $d(u,v)$, is the length of the shortest path between them in $G$. The \textit{eccentricity} of a vertex $v$ is $ecc(v) := \max_{u \in V(G)} d(u,v)$. The \textit{diameter} of $G$ is $diam(G) := \max_{u,v \in V(G)} d(u,v)$, and the \textit{radius} of $G$ is $rad(G) := \max_{v\in V(G)} ecc(v)$. The distance from a vertex $v$ to a set $X\subset V(G)$ is $d(v,X) = \min{x \in X} d(v,x)$. The \textit{$l$-step open neighborhood} of a set $X \subset V(G)$ is $N^l(X) = \{v \in V(G) | d(v, X) = l\}$, and the \textit{$l$-step neighborhood} of $X$ is $N^l[X] = \{v \in V(G) | d(v, X) \leq l\}$. A set $D^l \subset V(G)$ is called an \textit{$l$-step dominating set of $G$} if every vertex in $G$ is in the $l$-step neighborhood of $D^l$, and if $G|_D$ is a connected subgraph of $G$, then $D^l$ is called a \textit{connected $l$-step dominating set of $G$}. The \textit{connected $l$-step domination number} of a graph $G$, $\gamma_c^l(G)$, is the minimum possible size for a connected $l$-step dominating set of $G$. 

\begin{defn}
\end{defn} \vspace{-23pt} $~~~~~~~~~~~~~~~~~$ A \textit{planar graph} is a graph $G$ that can be embedded in the plane so as to preserve incidences and prevent crossing edges. Such an embedding of a graph in the plane is known as a \textit{planar embedding} of the graph. A \textit{face} of a planar graph is a minimal region bounded by edges. A \textit{maximal planar graph} $G$ is a planar graph where all faces are triangles.\\

\noindent Li et al. proved that Conjecture \ref{mainconj} holds for high girth graphs.
\begin{thm}[Li et al. \cite{li12}]
Let $G$ be a $\kappa$-connected graph. 
\begin{itemize}
\item[(i)] If $\kappa \geq 3$ and $girth(G) \geq 7$, then $rc(G) \leq \frac{n}{\kappa} + 41$.
\item[(ii)] If $\kappa \geq 5$ and $girth(G) \geq 5$, then $rc(G) \leq \frac{n}{\kappa} +19$.
\end{itemize}
\end{thm}

\noindent Hence it remains to show that Conjecture \ref{mainconj} holds for low girth graphs.

\section{Bounds on $rc(G)$ for Graphs with Large Diameter}
Let $G$ be a graph on $n$ vertices with vertex connectivity $\kappa$. Then $rc(G) \geq diam(G)$, and as $G$ is $\kappa$-connected, the diameter of $G$ is at most $\frac{n}{\kappa} +1$. Hence for graphs with large diameter, the lower bound is close to the conjectured upper bound. We show that, in the cases $\kappa=3$ and $\kappa=4$, $rc(G)$ is close to $\frac{n}{\kappa}$ for graphs with diameter close to $\frac{n}{\kappa}$.

\begin{thm}
Let $G$ be a $\kappa$-connected graph on $n$ vertices and let $diam(G) = \frac{n}{\kappa}-c$, where $c \geq 0$. 
\begin{itemize}
\item[(i)] If $\kappa = 3$, then $rc(G) \leq \frac{n}{3}+11c +6$.
\item[(ii)] If $\kappa = 4$, then $rc(G) \leq \frac{n}{4}+15c + 18$.
\end{itemize}
\end{thm}
\begin{proof}[Proof of (i)]$~$\\
Let $u_1, u_2 \in V(G)$ be a pair of vertices of $G$ such that $d(u_1,u_2) = diam(G)$. As $G$ is $3$-connected, there exist three internally vertex disjoint $u_1-u_2$ paths, $P_1$, $P_2$ and $P_3$. Note that we may choose the $P_i$ to be induced subgraphs, and we may assume without loss of generality that $|P_1| \leq |P_2| \leq |P_3|$. Let $X_i = V(P_i)-\{u_1,u_2\}$ be the interior vertices of each path, and let $X = X_1 \cup X_3 \cup \{u_1, u_2\}$. Let $\{Y_j\}$ be the vertex sets of the connected components of $G - \cup_i P_i$. 

Let $G'$ be the graph obtained from $G$ by contracting each component $G|_{Y_j}$ to a single vertex $y_j$, and let $Y = \{y_j\}$ be the set of vertices obtained by this contraction. It is clear that $G'$ is still 3-connected. This means, in particular, that any vertex $v \in X_i$ is incident with some vertex $w \not\in X_i$, as we have chosen the $P_i$ to be induced subgraphs of $G$.  

Color the cycle $C=u_1 \rightarrow P_1 \rightarrow u_2 \rightarrow P_3 \rightarrow u_1$ using $\left\lceil\frac{1}{2}(|P_1|+|P_3|-2)\right\rceil$ colors, in cyclic order $c_1, c_2, \ldots, c_{m}, c_1, \ldots$, so that $G$ restricted to $C$ is rainbow colored. Color the edges of the path $P_2$ using $|P_2|-1$ colors, one for each edge, reusing as many of the colors $\{c_1, \ldots, c_{m}\}$ as possible. Following this, all edges in the paths $P_i$ are now colored using $m := \max\left\{\left\lceil\frac{1}{2}(|P_1|+|P_3|-2)\right\rceil, |P_2|-1\right\}$ colors.

We now color the edges in $E(G') - \cup_i E(P_i)$. For each vertex of $G'$ not in $X$, we let $l(v)$ be the minimum length of a path from $v$ to $X$, with the constraint that the path does not contain adjacent edges in $X_2$ and uses the fewest edges of $P_2$ possible for any $v\rightarrow X$ path satisfying this constraint. As $G'$ is 3-connected, $l(v)$ is well-defined for every vertex $v \in X_2 \cup Y$.

Let $p(v)$ be a fixed vertex adjacent to $v$ such that $l(p(v)) < l(v)$. If $v \in X$, we take $l(v) = 0$. Hence for each $v$, the vertices $v, p(v), p(p(v)), \ldots, p^{(l(v))}(v)$ form a minimal $v \rightarrow X$ path that uses as few edges of $P_2$ as possible. Let $E_p$ be the set of edges of the form $\{v, p(v)\}$ for some $v \in V(G)-X$.  

Now color each edge $e \not \in \cup_i P_i$ and recolor some of the edges $e \in P_2$ by the color $c_e$, where $c_e$ is given by
\vspace{-5pt}
\begin{eqnarray*}
c_e = \left\{
\begin{array}{cc}
c_{m + l(v)}	&	\text{ if } e = \{v, p(v)\} \text{ and } v \in X_2 	\\
c_{v}				&	\text{ if } e = \{v, p(v)\} \text{ and } v \in Y 	\\
d				&	\text{ if } e = \{u,v\}\not\in E_p \text{ and } u \in Y, v \in X \cup \{u_1, u_2\}	\\
e				&	\text{ if } e = \{u,v\}\not\in E_p \text{ and } u \not\in X_2, v \in X_2.
\end{array}
\right.
\end{eqnarray*}

Then each of the edges of $G'$ are colored using one of $m + \max_{v} l(v) + |Y|  + 2$ colors. 

We may bound $\max_{v} l(v)$ as follows. Each $v \in Y$ is incident only to vertices in $X \cup X_2$. Moreover, if $v \in X_2$ and $l(v) > 1$ then $v$ is incident to a vertex in $Y$. Finally, there is no pair of adjacent edges in any path $v \rightarrow p(v) \rightarrow p(p(v)) \cdots p^{(l(v))}(v)$ that are both in $P_2$. Hence, if $D_l$ is the set of vertices with domination distance $l$, then every three sets $D_l, D_{l+1}, D_{l+2}$ contains at least one member of $Y$. Hence $\max_v l(v) \leq 3|Y|$. 

Finally, we show that such a coloring, applied after first coloring the edges of $P_1 \cup P_3$, is a rainbow coloring of $G'$. We first note the following useful facts.

\begin{itemize}

\item \underline{If $u, v \in X$ then there is a rainbow $u \rightarrow v$ path using only colors in $\{c_1, \ldots, c_{m}\}$.}\\
It is easy to verify that the cycle $G'|_{P_1 \cup P_3}$ is rainbow colored. Hence there is a rainbow $u\rightarrow v$ path contained in $P_1 \cup P_3$.

\item \underline{If $u, v \in X_2$ then there is a rainbow $u \rightarrow v$ path using only colors in $\{c_1, \ldots, c_{3|Y|}\}$.}\\
It is easy to verify that $P_2$ is a rainbow path, and hence any subpath is also a rainbow path. Moreover, as $\max_v l(v) \leq 3|Y|$, all the colors used are in the set $\{c_1, \ldots, c_{3|Y|}\}$.

\item \underline{If $u \not\in X$ then there is a rainbow $u \rightarrow X$ path using only colors in $\{c_v\}$.}\\
This is also clear, as the $u \rightarrow X$ path $u, p(u), \ldots, p^{(l(u))}(u)$ is such a path. 

\item \underline{If $u \not\in X$ then there is a rainbow $u \rightarrow X_2 \cup X$ path using only colors in $\{d, e\}$.}\\
This follows from 3-connectivity of $G'$. For if $u$ is a vertex of $G$, either $u \in X_2$ and the $u \rightarrow X_2$ path exists trivially, or $u \in Y$ and is adjacent only to vertices in $X$ and $X_2$. Since $G'$ is 3-connected, $u$ is incident with at least one edge not in $E_p$, which forms the required $u \rightarrow X_2 \cup X$ path.
\end{itemize}

Let $u,v$ be a pair of vertices in $G'$. We show that there is a rainbow $u\rightarrow v$ path. 

\begin{itemize}

\item \underline{Case 1: Both $u$ and $v$ are in $X$, or both $u$ and $v$ are in $X_2$.} These cases have been shown above.

\item \underline{Case 2: Exactly one of $u$ and $v$ is in $X$.}\\
We may assume without loss of generality that $v \in X$, $u \not\in X$. Then there exists some $w \in X$ such that there is a $u \rightarrow w$ path using only colors in $\{c_v\}$, and as $w \in X$ there is a $w \rightarrow v$ path using colors in $\{c_1, \ldots, c_{m}\}$, so concatenating these gives the required rainbow $u\rightarrow v$ path.

\item \underline{Case 3: Both $u$ and $v$ are not in $X$.}\\
If either $u$ or $v$ is adjacent, via an edge not in $E_p$, to a vertex in $X$, then we may choose the following rainbow $u\rightarrow v$ path. Without loss of generality let $v$ be adjacent to $v' \neq p(v)$ in $X$, so that there is a $v'\rightarrow v$ path using only the colors in $\{d,e\}$. As $u \not\in X$, there is also a $u \rightarrow X$ path using only colors in $\{c_v\}$ ending at some vertex $u'$. Finally, there is a $u' \rightarrow v'$ path using only colors in $\{c_1, \ldots, c_{m}\}$. Concatenating these three paths, we obtain a rainbow $u \rightarrow u' \rightarrow v' \rightarrow v$ path.

Otherwise the 1-step neighborhoods of $u$ and $v$ contain only vertices in $X_2 \cup Y$. Let $u' = p(u)$ and $v'$ be the end of a rainbow $v \rightarrow X \cup X_2$ path using only colors in $\{d,e\}$, where we may choose both $u'$ and $v'$ to be in $X_2$. Then we may concatenate the edge $u \rightarrow u'$, colored by the color $c_u$, with a $u' \rightarrow v'$ path using only colors in $\{c_1, \ldots, c_{3|Y|}\}$ and the $v'\rightarrow v$ path using only colors in $\{d,e\}$ to obtain a $u\rightarrow v$ rainbow path.
\end{itemize}

Hence $rc(G') \leq m + 4|Y| + 2$.

If we then take a spanning forest of $G - \cup_i P_i$ and color each of its edges a different color, we obtain a rainbow coloring of $G$ using at most $m+ 4|Y| + 2 + \sum_i(|Y_i|-1) \leq m+ 3|Y|+\sum_i |Y_i|+2$ colors. A bit of arithmetic yields
\begin{eqnarray*}
\left\lceil\frac{1}{2}(|P_1|+|P_3|-2)\right\rceil + \left\lfloor\frac{1}{2}(|P_1|+|P_3|-2)\right\rfloor + (|P_2|-1) - 1&=& n - \sum_i |Y_i| 
\end{eqnarray*}
\vspace{-15pt}
\begin{eqnarray*}
m + 11diam(G) - 4 &\leq& 4n - 4\sum_i |Y_i| \\
m + 4\sum_i |Y_i| &\leq& \frac{n}{3} +4 + 11c \\
rc(G) ~\leq~ m+ 3|Y| +\sum_i |Y_i|+2 &\leq& \frac{n}{3} + 11c + 6.
\end{eqnarray*}
Hence $rc(G) \leq \frac{n}{3} + 11c + 6$, as required.
\end{proof}

\begin{proof}[Proof of (ii)]
As before, take $u_1, u_2 \in V(G)$ a pair of vertices with distance $d(u_1, u_2) = diam(G)$, and four internally vertex disjoint $u_1\rightarrow u_2$ paths $P_1, P_2, P_3, P_4$. Similarly define $X_i$, define $\{Y_i\}$ as the vertex sets of the components of $G - \cup_i P_i$, let the graph $G'$ be obtained from $G$ by contracting each set $Y_i$ to a single point, and define $Y$ as before. 

Color the paths $P_1$ and $P_4$ in order from $u_1$ to $u_2$ using the colors $\{c_1, \ldots, c_{|P_i|-1}\}$ and color the paths $P_2$, $P_3$ in reverse order, from $u_2$ to $u_1$. 

Let $Z_i$ be the set of vertices in $Y$ and only adjacent to members of $X_i$, and for each $v \in X_i \cup Z_I$ let $l_i(v)$ to be the minimum length of a path from $v$ to $X_1$, using as few of the edges in $P_i$ as possible. Let $E_{p_i}$ be the set of all edges of the form $\{v, l_i(v)\}$. 

Now color each edge $e \not\in \cup_i P_i$ and recolor some of the edges $e \in \cup_i P_i$ by the color $c_e$, where $c_e$ is given by
\vspace{-5pt}
\begin{eqnarray*}
c_e = \left\{
\begin{array}{cc}
c_{i,l_i(v) }		&	\text{ if } e = \{v, l_i(v)\} \text{ and } v \in X_i 	\\
c_{i, v}				&	\text{ if } e = \{v, l_i(v)\} \text{ and } v \in Z_i		\\
d_i				&	\text{ if } e = \{u,v\}\not\in E_{p_i}  \text{ and } u \in Y, v \in X_i	\\
e_{i,j}			&	\text{ if } e = \{u,v\}\not\in E_{p_i} \text{ and } u \in X_i, v \in X_j.
\end{array}	
\right.
\end{eqnarray*}

As in the proof of (i), for any pair of vertices in $u \in X_j, v \in X_k$, unless $j + k = 5$, there is clearly a $u\rightarrow v$ rainbow path along the $P_i$. For any $u \in Z_i$ there is a rainbow $u \rightarrow \cup_{j \neq i} X_j$ path using only colors in $\{c_{i,v}\}$ and a rainbow $u \rightarrow X_i$ edge using the color $\{d_i\}$. Following a similar case analysis as in (i), concatenating rainbow paths within $X_i \cup Z_i$ with rainbow paths outside of $X_i \cup Z_i$, it can be shown that the assigned colors form a rainbow coloring of $G'$.

Again as in the proof of (i), it can also be seen that this rainbow coloring uses at most $\max_i (|P_i|-1) + \sum_i \max_v l_i(v) + \sum_i |Z_i| + 4 + {4 \choose 2}$ colors, where $\max_v l_i(v) \leq 3|Z_i|$ for each $i$. Hence
\begin{eqnarray*}
rc(G') &\leq& \max_i (|P_i|-1) + 4\sum_i |Z_i| + 10~~~~~
\end{eqnarray*}
\vspace{-17pt}
\begin{eqnarray*}
rc(G) &\leq& \left(\max_i (|P_i|-1) + 4\sum_i |Y_i| \right)+ 10\\
&\leq& \left(4n - 15(diam(G)) +8 \right) + 10 \\
&=& 4n - 15\left(\frac{n}{4}-c\right) +18 \\
&=& \frac{n}{4} + 15c + 18.
\end{eqnarray*}
This completes the proof.
\end{proof}

\section{Bounds on $rc(G)$ for Maximal Planar Graphs}
In this section, we prove that if $G$ is a maximal planar graph, then $rc(G) \leq \frac{n}{\kappa} + 36$. We remark that the proof can be quite simply extended to show that $rc(G) \leq \frac{n}{\kappa} + C_f$ for any planar graph graph with maximum face size $f$ and constants $C_f$ which grow at least linearly in $f$.

\begin{thm}\label{thm:max}
Let $G$ be a $\kappa$-connected maximal planar graph on $n$ vertices.
\begin{itemize}
\item[(i)] If $\kappa = 3$, then $rc(G) \leq \frac{n}{3} + 16$.
\item[(ii)] If $\kappa = 4$, then $rc(G) \leq \frac{n}{3} + 25$.
\item[(iii)] If $\kappa = 5$, then $rc(G) \leq \frac{n}{3} + 36$.
\end{itemize}
\end{thm}

We note that any planar graph has a vertex of degree at most $5$, and therefore has vertex connectivity at most $5$. Thus Theorem \ref{thm:max} proves Conjecture \ref{mainconj} for $C=36$ and all maximal planar graphs.

In proving the theorem, we make use of the following lemmas.

\begin{lem}[Basavaraju et al. \cite{bas}]\label{lem:bas}
Let $G$ be a bridgeless graph and let $D^l$ be a connected $l$-step dominating set of $G$. Then
\begin{eqnarray*}
rc(G) \leq rc(D^l)+l^2+2l.
\end{eqnarray*}
\end{lem}

\begin{lem}\label{lem:2conn_neighbs}
Let $G$ be a planar embedding of a maximal planar graph. Then for any cycle $C$, $N^1(C)$, the one-step neighborhood of $C$, is comprised of two 2-connected components, one inside $C$ and one outside $C$.
\end{lem}
\begin{proof}
By symmetry, it suffices to show that the component outside of $C$ is 2-connected.

Let $v_i$ be a vertex of $C$ and let $v_{i,1}, v_{i,2}, \ldots, v_{i,m_i}$ be the vertices outside of $C$ adjacent to $v$, in counterclockwise order. Then as $G$ is maximal planar, $v_{i,j}$ and $v_{i,j+1}$ are adjacent for all $i, 1 \leq j < m_i$. 

Let $v_i, v_{i+1}$ be two adjacent vertices in $C$, in counterclockwise order. Then as $G$ is maximal planar, $v_{i,m_i} = v_{i+1,1}$. 

Hence the subgraph induced by those vertices of $N^1(C)$ outside of $C$ is Hamiltonian, and hence 2-connected.
\end{proof}

\begin{proof}[Proof of Theorem \ref{thm:max}]$~$\\
\noindent Let $F$ be the vertex set of an arbitrary face of $G$, and let $N_k = N^k(F)$ be the $k$-step open neighborhood of $F$. Let $t$ be maximal such that $N_t$ is non-empty. 

We construct a connected $\kappa$-step dominating set $D$ of $G$ and show that $rc(D) \leq \frac{n}{\kappa} + 1$. Then by Lemma \ref{lem:bas}, $rc(G) \leq rc(D) + \kappa^2+2\kappa = \frac{n}{\kappa} + 1 + \kappa^2 + 2\kappa$ and the theorem is proved.

Let $A_1 = \{k : |N_k| \leq 2\kappa-1\}$, $A_2 =\{k: |N_k| = 2\kappa\}$ and $A_3 =\{k: |N_k| \geq 2\kappa+1\}$, with cardinalities $n_i=|A_i|$. We label the elements of $A_2 \cup A_3$ in ascending order, $A_2 \cup A_3 = \{k_1, \ldots, k_{|A_2 \cup A_3|}$. We choose a path $P$ of length $t$ containing exactly one vertex from each of the $N_k$. We also choose $A = \{k_i: i \equiv a ~(mod~ \kappa)\}$ to be the congruence class, modulo $\kappa$, such that $\sum_{k \in A} \left\lceil \frac{|N_k|}{2} \right\rceil$ is minimized. Let $D = F \cup \{N_k : k \in A\}  \cup V(P)$. 

We show first that $D$ is a connected $l$-step dominating set of $G$. Connectivity is clear, as $P$ is connected, $N_k$ is connected for each $k$ and $P \cap N_k$ is non-empty for each $k$. To show that $D$ is an $l$-step dominating set of $G$, consider some $v \in N_k$. Let $k' \leq k$ be maximal such that either $|N_{k'}| \leq 2 l_1+1$, or $k' \in A$. Let $v' \in N_{k'}$ be chosen such that $d(v,v') = k-k'$. Then $d(v, D) \leq d(v,v') + d(v',D)$. Note that $k-k' < l_2$. 

If $k'$ is such that $|N_{k'}| \leq 2 l_1+1$, then there is a vertex $v_{k'} \in N_{k'} \cap D$, so $d(v', D) \leq d(v', v_{k'}) \leq l_1$, as $N_{k'}$ is 2-connected. Hence $d(v, D) \leq l_2-1+l_1 = l$. If $k'$ is such that $k'=k_i$ for some $i$, then $v' \in D$ so $d(v, D) = d(v,v') < l$. Hence $D$ is a connected $l$-step dominating set of $G$.

We now show that $rc(D) \leq \frac{n}{\kappa}$ by constructing a rainbow coloring of $D$. Rainbow color $P$ using $|P|$ colors, and as each $N_k$ is 2-connected, rainbow color each $N_k$ using at most $\left\lceil \frac{|N_k|}{2} \right\rceil$ colors. As new colors are used for each $N_k$, this clearly gives a rainbow coloring of $D$.

Moreover, the rainbow coloring uses at most $|P| + \sum_{k \in A} \left\lceil \frac{|N_k|}{2} \right\rceil$ colors. Hence
\vspace{-5pt}
\begin{eqnarray*}
rc(D) &\leq& (n_1+n_2+n_3-1) + \frac{1}{\kappa}\sum_{k \in A_2 \cup A_3} \left\lceil \frac{|N_k|}{2} \right\rceil ~~~~~~~~~~~~~\text{   (by choice of $A$)}\\
&\leq& (n_1+n_2+n_3-1) + \frac{1}{2\kappa}\sum_{k \in A_2} |N_k| + \frac{1}{2\kappa}\sum_{k \in A_3} \left(|N_k| + 1\right)
\end{eqnarray*}
\begin{equation}\label{eqn:rc}
rc(D) ~\leq~ n_1+\frac{2\kappa n_2+ (2\kappa + 1)n_3}{2\kappa}-1 + \frac{1}{2\kappa}\sum_{k \in A_2 \cup A_3} |N_k|. ~~~~~~~~~~~~~~~~~~~~~~~~~
\end{equation}
To simplify this, we note that if $k \in A_2$ then $|N_k| = 2\kappa$, and if $k \in A_3$ then $|N_k| \geq 2\kappa +1$. Hence

\begin{equation}\label{eqn:A2A3}
\sum_{k \in A_2 \cup A_3}|N_k| ~\geq~ 2\kappa n_2 + (2\kappa +1)n_3.
\end{equation}

Moreover, if $k\neq 0, t \in A_1$ then $|N_k| \geq \kappa$, we also have $|N_0| = 3$ and $|N_t| \geq 1$. Hence

\begin{equation}\label{eqn:ns}
\sum_{k \in A_2 \cup A_3} |N_k| ~=~ n - \sum_{k \in A_1} |N_k| ~\leq~ n - \kappa n_1 + 2\kappa - 4.
\end{equation}
Putting together equations (\ref{eqn:rc}), (\ref{eqn:A2A3}) and (\ref{eqn:ns}), we obtain
\begin{eqnarray*}
rc(D) ~\leq~ n_1 + \frac{1}{\kappa}\sum_{k \in A_2 \cup A_2} |N_k| -1 ~\leq~ n_1 + \frac{1}{\kappa} (n - \kappa n_1 + 2\kappa -4) - 1 ~\leq~ \frac{n}{\kappa} + 1.
\end{eqnarray*}
This gives the required bound on $rc(D)$ and completes the proof of the theorem.

\end{proof}

We note that Theorem \ref{thm:max} can be extended to planar graphs with maximum face size $f$, by defining the $N_k$ recursively, taking $N_{k+1}$ to be a 2-connected subset of $N^{\left\lceil\frac{f-1}{2}\right\rceil}(N_k)$-step open neighborhood of $N_k$. However, the constant $C$ grows at least linearly with $f$, so this would not lead to a proof of the general case. 

\section{Acknowledgements}

This research was performed at the University of Minnesota Duluth REU run by Joe Gallian. The REU was supported by the National Science Foundation, the Department of Defense (grant number DMS 1062709), the National Security Agency (grant number H98230-11-1-0224) and Princeton University. 

I would like to thank Joe Gallian for supervising this research. I am also grateful to Eric Riedl, David Rolnick, Samuel Elder and the other participants at the 2012 Duluth REU for their insightful comments and support. 

\bibliographystyle{plain}	
\bibliography{RainbowRefs}		

\end{document}